\documentclass[dvips,sts]{imsart}

\usepackage{amsthm,amsmath, amssymb}
\RequirePackage[colorlinks,citecolor=blue,urlcolor=blue]{hyperref}

\RequirePackage{hypernat}

\font\large=cmr10 scaled 1800
\font\tenmath=msbm10
\font\sevenmath=msbm7
\font\fivemath=msbm5
\newfam\mathfam \textfont\mathfam=\tenmath
\scriptfont\mathfam=\sevenmath \scriptscriptfont\mathfam=\fivemath

\usepackage{amsmath, moreverb, hyperref}
\hypersetup{colorlinks=true,citecolor=blue}

\usepackage{color}
\usepackage{amsfonts, fancybox}
\usepackage{amssymb}
\usepackage[american]{babel}
\usepackage{graphicx, pstricks, pst-node, color}
\usepackage{psfrag}\usepackage{subfigure}
\usepackage{flafter, bm, dsfont}
\usepackage[section]{placeins}

\startlocaldefs

\renewcommand{\phi}{\varphi}

\renewcommand{\P}{{\rm I}\kern-0.18em{\rm P}}
\newcommand{\E}{{\rm I}\kern-0.18em{\rm E}}

\newcommand{\R}{{\rm I}\kern-0.18em{\rm R}}

\newcommand{\KL}{\mathrm{KL}}

\newcommand{\cN}{\mathcal{N}}

\newcommand{\cX}{\mathcal{X}}

\newcommand{\dsone}{\mathds{1}}
\renewcommand{\epsilon}{\varepsilon}

\newcommand{\wh}{\widehat}
\newcommand{\argmax}{\mathop{\mathrm{argmax}}}

\newlength{\minipagewidth}
\newcommand{\bookbox}[1]{
\par\medskip\noindent
\framebox[\textwidth]{
\begin{minipage}{\minipagewidth}
{#1}
\end{minipage} } \par\medskip }

\newcommand{\beq}{\begin{equation}}
\newcommand{\eeq}{\end{equation}}

\newcommand{\beqa}{\begin{eqnarray}}
\newcommand{\eeqa}{\end{eqnarray}}

\newcommand{\beqan}{\begin{eqnarray*}}
\newcommand{\eeqan}{\end{eqnarray*}}

\def\ba#1\ea{\begin{align*}#1\end{align*}} 
\def\banum#1\eanum{\begin{align}#1\end{align}} 

\newcommand{\bi}{^{(i)}}
\newcommand{\bpi}{^{(I_t)}}

\newcommand{\BlackBox}{\rule{1.5ex}{1.5ex}}
\renewenvironment{proof}{\par\noindent{\bfseries\upshape
  Proof\ }}{\hfill\BlackBox\\[2mm]}

\newtheorem{theorem}{Theorem}
\newtheorem{lemma}[theorem]{Lemma}

\begin{document}

\begin{frontmatter}

\title{Bounded regret in stochastic multi-armed bandits}
\runtitle{Bounded regret in stochastic bandits}

 \author{\fnms{S\'ebastien} \snm{Bubeck}\ead[label=e1]{sbubeck@princeton.edu}, \fnms{Vianney} \snm{Perchet}\corref{}\thanksref{t2}\ead[label=e2]{vianney.perchet@normalesup.org}\and \fnms{Philippe} \snm{Rigollet}\corref{}\thanksref{t3}\ead[label=e3]{rigollet@princeton.edu}}

\thankstext{t2}{Partially supported by the ANR (ANR-10-BLAN-0112).}
\thankstext{t3}{Partially supported by NSF grants DMS-0906424, CAREER-DMS-1053987 and a gift from the Bendheim Center for Finance.}

 \affiliation{Princeton University, Universit\'e Paris Diderot and Princeton University}

 \address{{S\'ebastien Bubeck}\\
{Department of Operations Research} \\
{ and Financial Engineering}\\
{Princeton University}\\
{Princeton, NJ 08544, USA}\\
 \printead{e1}}
\address{{Vianney Perchet}\\
{LPMA, UMR 7599}\\
{Universit\'e Paris Diderot}\\
{175, rue du Chevaleret}\\
 {75013 Paris, France}\\
 \printead{e2}}
 \address{{Philippe Rigollet}\\
{Department of Operations Research} \\
{ and Financial Engineering}\\
{Princeton University}\\
{Princeton, NJ 08544, USA}\\
 \printead{e3}}

%

\runauthor{Bubeck, Perchet and Rigollet}

\begin{abstract}
\ We study the stochastic multi-armed bandit problem when one knows the value $\mu^{(\star)}$ of an optimal arm, as a well as a positive lower bound on the smallest positive gap $\Delta$. We propose a new randomized policy that attains a regret {\em uniformly bounded over time} in this setting. We also prove several lower bounds, which show in particular that bounded regret is not possible if one only knows $\Delta$, and bounded regret of order $1/\Delta$ is not possible if one only knows $\mu^{(\star)}$.
\end{abstract}



\end{frontmatter}

\section{Introduction}
In this paper we investigate the classical stochastic multi-armed bandit problem introduced by \cite{Rob52} and
described as follows: an agent facing $K$ actions (or bandit arms) selects one arm at every time step until a finite time horizon $n \ge 1$. Successive pulls of each arm $i \in \{1, \hdots, K\}$ yield a sequence of i.i.d rewards $Y\bi_1, Y\bi_2, \ldots$ according to some unknown distribution $\nu_i$ with expected value $\mu\bi$. Denote by $\star \in \{1,\ldots, K\}$ any optimal arm defined such that $\mu^{(\star)} = \max_{i =1, \ldots, K} \mu\bi$. A \emph{policy} $I=\{I_t\}$ is a sequence of random variables $I_t \in\{1, . . . , K \}$ indicating which arm to pull at each time $t=1, \ldots, n$ and such that $I_t$ depends only on observations strictly anterior to $t$. The performance of a policy $I$ is measured by its (cumulative) \emph{regret} at time $n$ that is defined by
$$
R_n = n \mu^{(\star)} - \sum_{t=1}^n \E\,\mu^{({I_t})}\,.
$$
Observe that if we denote by $T_i(t) = \sum_{\ell=1}^{t-1} \dsone\{I_\ell = i\}$ the number of times arm $i$ was pulled (strictly) before time $t \geq 2$ and by $\Delta_i = \mu^{(\star)} - \mu\bi$ the gap between arm $i$ and the optimal arm, then one can rewrite the regret as $R_n = \sum_{i=1}^K \Delta_i \E T_i(n+1)$. This formulation will be   used hereafter.

We refer the reader to \cite{BC12} for a survey of the extensive literature on this problem and its variations. In this paper we investigate a phenomenon that was first observed in \cite{LR84}:
with some prior knowledge 
(in the form of lower bounds) on the maximal mean $\mu^{(\star)}$ and the minimal gap $\Delta = \min_{i : \Delta_i > 0} \Delta_i$, it is possible to obtain a regret that is {\em bounded uniformly in $n$}, which implies in particular that the regret does not tend to infinity as the time horizon $n$ tends to infinity. Note that this result is striking, as the seminal paper \cite{LR85} indicates that, if one has no prior knowledge on the distributions, then asymptotically (in $n$) a regret of order $\log n$ is unavoidable.

\subsection{Contributions}
We describe in Section~\ref{sec:2armed} a simple algorithm for the two-armed bandit problem when one knows the largest expected reward $\mu^{(\star)}$ and the gap $\Delta$. In this two-armed case, this amounts to knowing $\mu^{(1)}$ and $\mu^{(2)}$ up to a permutation. We show that the regret of this algorithm is bounded by $\Delta + 16/\Delta$, uniformly in $n$.  The optimality of this bound is assessed in Section~\ref{sec:LB} where we show that any agent knowing $\Delta$ and $\mu^{(\star)}$ must incur a regret of at least $1/\Delta$. This upper and lower bounds raise the following question: can such bounded regret be achieved without one of these two pieces of information? It follows from Theorems~\ref{th:LB2} and~\ref{th:LB3} that the answer to this question is negative. Indeed, the sole knowledge of either $\Delta$ or $\mu^{(\star)}$ 
leads to a rescaled regret $\Delta R_n$ that is at least logarithmic in $n$. Interestingly, all these results are fully non-asymptotic, including lower bounds.

What if $\Delta$ is not perfectly known but only $\epsilon>0$ such that $\Delta>\epsilon$? We answer this question in Section~\ref{sec:general} in the context of the general $K$-armed bandit problem. There, we prove an upper bound on $R_n$ when one knows the maximal mean $\mu^{(\star)}$ together with a positive lower bound $\epsilon$ on the smallest gap $\Delta$. Specifically, we design a randomized policy for which
$$
R_n \leq \sum_{i: \Delta_i > 0} \Big\{ \Delta_i+\frac{32}{\Delta_i}\log\big(\frac{5}{\epsilon}\big) \Big\}\,.
$$ 
Moreover, it follows form our main lower bound in Theorem~\ref{th:LB3}  that this result cannot be improved without further assumptions, since for $\epsilon$ of order of $1/\sqrt{n}$ ---no information on the smallest gap--- a logarithmic growth in $n$ is unavoidable for the rescaled regret $\Delta R_n$. However for $\epsilon$ of order $\Delta$ one would expect no dependency on $\epsilon$ (since at least for $K=2$ our policy of Section \ref{sec:2armed} attains a regret of order $1/\Delta$). To deal with this issue we propose an improvement of the basic policy that for which the term $\log(1/\epsilon)$ is replaced by $\log(\Delta_i / \epsilon) \log\log \epsilon$. In particular if all the gaps $\Delta_i$ and $\epsilon$ are of the same order, the logarithmic becomes a log-log term.

The \emph{exploration-exploitation tradeoff} is a preponderant paradigm in the bandit literature. The effects of this tradeoff already appear for the case $K=2$ in the form of the $\log n$ term derived in the original \cite{LR85} paper. Indeed,  there 
exist simple classes of (two!) problems over which the regret is uniformly bounded with full information but cannot be bounded uniformly with bandit feedback, see Theorem~\ref{th:LB2}. Clearly, 
this tradeoff should become more and more apparent as the number of arms increases but this is not our main focus. Rather, the combination of our results sheds light on an interesting phenomenon:  the effects of the tradeoff vanish when both $\Delta$ and $\mu^{(\star)}$ are known but can be seen already when $K=2$ and either $\Delta$ or $\mu^{(\star)}$ is unknown.

\subsection{Related works}
The two-armed bandit problem when one knows the distributions of the arms up to a permutation was first investigated in \cite{LR84}. The authors observed that in that case, using a policy based on the sequential likelihood ratio test, one can obtain a regret uniformly bounded over $n$. Both upper and lower bounds were provided. This setting was generalized in \cite{LR84book}, where the authors considered the general multi-armed bandit problem when one knows a separating value $\gamma$ between the largest mean and the other means. In that case they proved the bounded regret property for a policy based on sequential likelihood ratio tests for $H_0\,:\, \mu>\gamma$ vs. $H_1\,:\, \mu < \gamma$ (assuming exponential distributions to compute the likelihoods). They also designed a more subtle strategy for the case when only $\mu^{(\star)}$ is known. In that case too they proved a bounded regret property. The main open problems left by these works are (i) to understand the limitations of bounded regret, and (ii) to characterize the exact dependence on the parameters in the regret (when bounded regret is achievable). In this paper we make progress on both questions.

Regarding the limitations of bounded regret, we prove three finite-time lower bounds, including a finite-time version of the seminal result of \cite{LR85}. Ideas similar to the ones we develop in Theorems~\ref{th:LB1} and~\ref{th:LB2} already appeared in \cite{KL00} but our results are fully non asymptotic with the exact dependence in the parameters involved. Theorem~\ref{th:LB3} is more innovative. It shows that a logarithmic growth for the rescaled regret $\Delta R_n$ is unavoidable even if one knows $\mu^{(\star)}$. The proof of this result goes beyond any previous lower bound for the stochastic multi-armed bandit problem, including \cite{LR84book, LR85}, since all of them required to distinguish problems with different values of $\mu^{(\star)}$ (such as the ones in Theorem~\ref{th:LB2} for example). As a consequence of this theorem, we can deduce that the policies with bounded regret derived in \cite{LR84book, AgrTenAna89} with only the knowledge of $\mu^{(\star)}$ must have a suboptimal dependency in $1/\Delta$.

The knowledge of $\mu^{(\star)}$ was also exploited in other works. For instance in \cite{SA11}, the authors showed that knowing $\mu^{(\star)}$ allows for policies with provably better concentration properties. Their policies are based on sequential likelihood ratio tests for $H_0\,:\, \mu=\mu^{(\star)}$ vs. $H_1\,:\, \mu < \mu^{(\star)}$ (assuming Gaussian distributions to compute the likelihoods). To some extent it was to be expected that the knowledge of $\mu^{(\star)}$ leads to an improved regret as it partially removes the need for exploration: if one arm has empirical performances close to $\mu^{(\star)}$, one can be confident that this is the best arm without worrying that it could be the best arm only because we have not yet explored enough the other options. 
However note that the problem turns out to be more subtle than the above simple argument and underlines the fact that one needs more than the knowledge of $\mu^{(\star)}$ in order to have a bounded regret with optimal scaling in $1/\Delta$.
Indeed, Theorem~\ref{th:LB3} implies that the sole knowledge of $\mu^{(\star)}$ does not warrant the bounded property for the rescaled regret $\Delta R_n$. 



\subsection{Basic assumptions} Throughout the paper, we assume that the distributions $\nu_i$ are sub-Gaussian that is $\int e^{\lambda(x-\mu)}\nu_i(dx)\le e^{\lambda^2/2}$ for all $\lambda \in \R$. Note that these include Gaussian distributions with variance less than $1$ and distributions
 supported on an interval of length less than $2$.

We denote  by $\wh{\mu}\bi_{s} = \frac{1}{s} \sum_{\ell=1}^s Y\bi_{\ell}$ the empirical mean of arm $i$ after $s$ pulls, for $s \geq 1$.
Together with a Chernoff bound, it is not hard to see that the sub-Gaussian assumption implies the following concentration inequality, valid for any $u>0$,
\begin{equation} \label{eq:hoe}
\P(\wh{\mu}\bi_{s} - \mu\bi > u) \leq \exp\left(- \frac{s u^2}{2}\right) .
\end{equation}
%
%
\section{The two-armed case} \label{sec:2armed}
In this section we investigate a toy example where $K=2$ and the agent knows exactly both $\mu^{(\star)}=0$ (without loss of generality) and $\Delta$. While somewhat simplistic this example offers a convenient framework to lay the main ideas to build policies with bounded regret.



\begin{figure}[h!]
\bookbox{
Initialization:  
\begin{itemize} \item[(0)] For rounds $t \in \{1,2\}$, select arm $I_t=t$.
\end{itemize}
For each round $t=3,4,\ldots$
\begin{itemize}
\item[(1)]
If $\wh{\mu}\bi_{T_i(t)} > - \Delta /2$ and $\wh{\mu}\bi_{T_i(t)} > \wh{\mu}^{(j)}_{T_j(t)}$ then select arm $i$, i.e., $I_t = i$. 
\item[(2)]
Otherwise select both arms, i.e., $I_t=1$ and $I_{t+1}=2$.
\end{itemize}}

\caption{\label{fig:alg1}
A policy with bounded regret for the two-armed bandit problem.}
\end{figure}

\begin{theorem} \label{th:alg1}
Policy~\ref{fig:alg1} has regret bounded as $R_n \leq \Delta + 16/\Delta$, uniformly in $n$.
\end{theorem}
\begin{proof}
Without loss of generality we assume that $1=\star$ is the optimal arm. Observe that
$$
\{I_t = 2\} \subset \{t=2\} \cup \{\wh{\mu}^{(2)}_{T_2(t)} > - \Delta /2 \,, t \geq 3, \, \ I_t = 2\} \cup \{\wh{\mu}^{(2)}_{T_2(t)} \leq - \Delta /2 \,, t \geq 3, \, \ I_t = 2\} .$$
Summing over $t$ for the 
second event, we get
\begin{equation} \label{eq:1}
\E \sum_{t=3}^n \dsone\{\wh{\mu}^{(2)}_{T_2(t)} > - \Delta /2 \,, \ I_t = 2\} \leq \E \sum_{t=1}^n \dsone \{\wh{\mu}^{(2)}_{t} > - \Delta /2\} \leq \sum_{t=1}^n \exp(- t \Delta^2 / 8) \leq \frac{8}{\Delta^2} .\end{equation}

For the third event we use the definition of the policy to obtain
$$\{\wh{\mu}^{(2)}_{T_2(t)} \leq - \Delta /2 \,, t \geq 3, \, \ I_t = 2\} \subset \{\wh{\mu}^{(1)}_{T_1(t-1)} \leq - \Delta /2 \,, t \geq 3, \, \ I_{t-1} = 1\}\, $$
and conclude as in~\eqref{eq:1}.
\end{proof}
This policy has two weaknesses. First one may pay a big price for misspecifying the value of~$\Delta$. Namely if one only knows a lower bound $0<\epsilon\le \Delta$ and substitutes $\epsilon$ to $\Delta$ in Policy~\ref{fig:alg1}, then it follows easily that 
the regret becomes of order $\Delta/\epsilon^2$. Furthermore, for essentially the same reason, the trivial generalization of this algorithm to the $K$-armed case would give a regret bounded by 
$\sum_i \Delta_i/\Delta^2$. In the next section we show how to overcome these two issues using a new, randomized, policy.

\section{A family of policies with bounded regret} \label{sec:general}
In this section we consider the general multi-armed case, when the agent knows $\mu^{(\star)}=0$ (without loss of generality) and an $\epsilon>0$ such that $\epsilon \leq \Delta$. Akin to Policy~\ref{fig:alg1}, the  policy analyzed here sets a threshold at $-\epsilon/2$ and prescribes to pull a single arm above this threshold. However if all arms have their empirical mean below this threshold, then the policy is more subtle than what was described in the previous section (where all arms were pulled in round robin fashion). Here the policy picks an arm at random, where the probability of selecting arm $i$ is essentially proportional to $(\hat\mu\bi_{T_i(t)})^{-2}$,
 which is an empirical estimate of $\Delta_i^{-2}$ since $\mu^{(\star)}=0$. Policy~\ref{fig:alg3} is slighly more general, as it uses a potential function $\psi:\R_+ \rightarrow \R_+$, and selects arm $i$ with probability inversely proportional to $\psi(|\wh{\mu}\bi_{T_i(t)}|)$. The natural choice is $\psi(x) = x^2$, but other choices can lead to improved performances, see Theorem \ref{th:algPsi} below. 
Note that we also analyze the case where $\epsilon = 0$ (that is, when we have no information on the smallest gap).


\begin{figure}[h!]
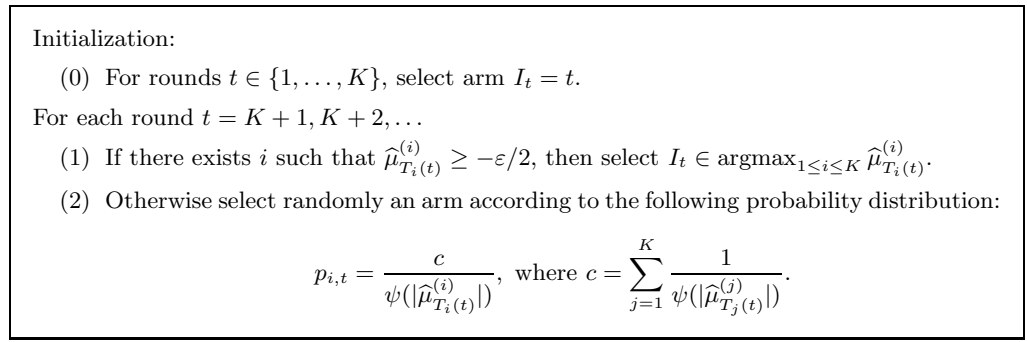

\bookbox{
Initialization:  
\begin{itemize} \item[(0)] For rounds $t \in \{1,\ldots,K\}$, select arm $I_t=t$.
\end{itemize}
For each round $t=K+1,K+2,\ldots$
\begin{itemize}
\item[(1)]
If there exists $i$ such that $\wh{\mu}\bi_{T_i(t)} \geq - \epsilon /2$, then select $I_t \in \argmax_{1 \leq i \leq K} \wh{\mu}\bi_{T_i(t)}$. 
\item[(2)]
Otherwise select randomly an arm according to the following probability distribution:
$$p_{i,t} = \frac{c}{\psi(|\wh{\mu}\bi_{T_i(t)}|)}, \ \text{where} \ c = \sum_{j=1}^K \frac{1}{\psi(|\wh{\mu}^{(j)}_{T_j(t)}|)} .$$
\end{itemize}}

\caption{\label{fig:alg3}
A family of policies with bounded regret for the $K$-armed bandit problem.}
\end{figure}

\begin{theorem}\label{th:algPsi}
Fix $\epsilon \in (0, 1 \wedge \Delta]$, then Policy~\ref{fig:alg3} associated with the potential $\psi(x)=x^2$ satisfies for all $n \geq 1$,
\begin{equation}
\label{eq:psisimp}
R_n \leq \sum_{i: \Delta_i > 0} \Big\{ \Delta_i+\frac{32}{\Delta_i}\log\big(\frac{5}{\epsilon}\big) \Big\}\,.
\end{equation}
Furthermore for $\epsilon = 0$, let $v = \E \left(Y_1^{(\star)}\right)^2$, then the regret is bounded as
\begin{equation}
\label{eq:psiepszero}
R_n \leq \sum_{i : \Delta_i >0}  \Big\{\Delta_i + (1 \vee v)\frac{4 \log(9 n)}{\Delta_i} \Big\}\,.
\end{equation}
The dependency in $\epsilon$ can be reduced by using the potential $\psi(x) = \frac{x^2}{\log(4 x/\epsilon)}$ since it yields
\begin{equation}\label{eq:psilog}
R_n \leq \sum_{i: \Delta_i > 0}\Big\{ \Delta_i+\frac{32\log\big(\frac{2\Delta_i}{\varepsilon}\big)}{\Delta_i}\big[3+   \log\log\big(\frac{4}{\epsilon}\big)\big]\Big\}\,.\end{equation}
\end{theorem}

If $\varepsilon$ is of the order of every $\Delta_i$, then Equation \eqref{eq:psilog}  upper bounds the regret in $\sum_i \log \log(1/\Delta_i)/\Delta_i$; on the other hand, using the potential $\psi(x)=x^2$ only guarantees, under the same assumptions, a bound in $\sum_i  \log(1/\Delta_i)/\Delta_i$.

The result for $\epsilon=0$ implies that when one has no information on the smallest gap, our policy does not obtain bounded regret but it recovers the performances of UCB, \cite{ACF02}. As we shall see in Section \ref{sec:LB} it is in fact impossible to obtain bounded regret scaling in $1/\Delta$ if one only knows $\mu^{(\star)}$. 

Theorem \ref{th:algPsi} is deduced from the following more general regret bound for Policy~\ref{fig:alg3} expressed in terms of the properties of the potential $\psi$. 

\begin{theorem}\label{th:alg3}
Fix $\epsilon \in [0, \Delta]$ and let $\psi$ be a differentiable and increasing function $\psi: [\epsilon/2, +\infty) \rightarrow \R^+$. If $\epsilon > 0$, Policy~\ref{fig:alg3} satisfies for all $n \geq 1$,
\begin{equation}
\label{EQ:mainUB}
R_n \leq \sum_{i: \Delta_i > 0} \Big\{\Delta_i + \frac{8}{\Delta_i} + \frac{\Delta_i}{\psi(\Delta_i/2)}\Big[ \frac{8\psi(\epsilon/2)}{\epsilon^2}+ \int_{\epsilon/2}^{+\infty} \frac{2\psi'(x)}{e^{\frac{x^2}{2}}-1} dx\Big]\Big\}\,.
\end{equation}
Furthermore for $\epsilon = 0$ it satisfies
\begin{equation}
\label{eq:psiepszero0}
R_n \leq \sum_{i: \Delta_i > 0} \Big(\Delta_i + \frac{8}{\Delta_i} + \frac{\Delta_i}{\psi(\Delta_i/2)} \sum_{t=1}^n \E \ \psi(|\wh{\mu}^{(1)}_{t}|) \Big).
\end{equation}
\end{theorem}
\begin{proof}
Without loss of generality we assume that $1=\star$ is the optimal arm. We decompose the event of a wrong selection into three events:
\begin{align*}
\{I_t = i\} \subset  \{t=i\}&\cup\{\wh{\mu}\bi_{T_i(t)} > - \Delta_i /2 \,, \, t \geq K+1,\, \ I_t = i\}\\
& \cup \{\wh{\mu}\bi_{T_i(t)} \leq - \Delta_i /2 \,, \, t \geq K+1,\, \ I_t = i\}\,.
\end{align*}

Using \eqref{eq:1} one can easily prove that the cumulative probability of the first 
two events is smaller than $1+8/\Delta_i^2$. For the 
third event, it is convenient to define the random variable $Z \in \{0,1,2\}$ that indicates whether the agent plays according to (0), (1) or (2) in Policy~\ref{fig:alg3}. We write the following, using the definition of the algorithm and the fact that $\psi$ is non-decreasing,
\begin{align*}
& \P \{\wh{\mu}\bi_{T_i(t)} \leq - \Delta_i /2 \,, \, t \geq K+1,\, \ I_t = i\}  =  \P \{\wh{\mu}\bi_{T_i(t)} \leq - \Delta_i /2 \,, \ I_{t} = i \,, Z=2\} \\
& =  \E \ p_{i,t} \dsone \{\wh{\mu}\bi_{T_i(t)} \leq - \Delta_i /2 \,, Z=2\}   =  \E \ \frac{p_{i,t}}{p_{1,t}} p_{1,t} \dsone \{\wh{\mu}\bi_{T_i(t)} \leq - \Delta_i /2 \,, Z=2\}  \\ 
& \leq  \E \ \frac{\psi(|\wh{\mu}^{(1)}_{T_1(t)}|)}{\psi(\Delta_i/2)} p_{1,t} \dsone \{\wh{\mu}\bi_{T_i(t)} \leq - \Delta_i /2 \,, Z=2\}   \leq  \frac{1}{\psi(\Delta_i/2)} \ \E \ \psi(|\wh{\mu}^{(1)}_{T_1(t)}|) p_{1,t} \dsone \{Z=2\}  \\
& \le  \frac{1}{\psi(\Delta_i/2)} \ \E \ \psi(|\wh{\mu}^{(1)}_{T_1(t)}|)  \dsone \{\wh{\mu}^{(1)}_{T_1(t)} <- \epsilon /2 \,, \, t \geq K+1\}  .\\
\end{align*}
A simple rewriting 
of time then concludes the proof for the case of $\epsilon = 0$. We use the slight abuse of notation $\psi^{-2}(x):=[\psi^{1}(x)]^2$, and $\psi(\infty) = \lim_{x \to +\infty} \psi(x)$. For $\epsilon > 0$ we have

\begin{align*}
&\sum_{t=1}^n \E \ \psi(|\wh{\mu}^{(1)}_{T_1(t)}|) \dsone\{\wh{\mu}^{(1)}_{T_1(t)} \leq - \epsilon /2 \} \leq  \sum_{t=1}^n \E \ \psi(|\wh{\mu}^{(1)}_{t}|) \dsone\{\wh{\mu}^{(1)}_{t} \leq - \epsilon /2\} \nonumber\\
& = \sum_{t=1}^n \int_0^{+\infty} \P \left( \psi(|\wh{\mu}^{(1)}_{t}|) \dsone\{\wh{\mu}^{(1)}_{t} \leq - \epsilon /2\} \geq x \right) dx \nonumber\\ 
& =  \sum_{t=1}^n \Big\{\psi\left(\frac{\epsilon}{2}\right)\P\big(|\wh{\mu}^{(1)}_{t}|>\frac{\epsilon}{2}\big)+\int_{\psi(\epsilon/2)}^{\psi(\infty)} \P(\psi(|\wh{\mu}^{(1)}_{t}|) \geq x) dx\Big\} \nonumber\\
& \leq  \sum_{t=1}^n \Big\{ \psi\left(\frac{\epsilon}{2}\right)e^{- \frac{t\epsilon^2}{8}} +\int_{\psi(\epsilon/2)}^{\psi(\infty)} 2 e^{- \frac{t\psi^{-2}(x)}{2} }dx \Big\}\\
&\le \frac{8}{\epsilon^2}\psi\left(\frac{\epsilon}{2}\right)+ \int_{\psi(\epsilon/2)}^{\psi(\infty)} \frac{2}{e^{\frac{\psi^{-2}(x)}{2}}-1} dx.\nonumber 
\end{align*}
Making the change of variable $x=\psi(u)$ concludes the proof of Theorem~\ref{th:alg3}.
\end{proof}


Theorem \ref{th:algPsi} follows from Theorem~\ref{th:alg3} with specific choices for $\psi$.
First, take $\psi(x) = x^2$, $\epsilon \in (0, 1]$ and observe that the integral in~\eqref{EQ:mainUB} can be computed as
$$
\int_{\epsilon/2}^{+\infty} \frac{4x}{e^{\frac{x^2}{2}}-1} dx=-4\log\big(1-e^{-\frac{\epsilon^2}{8}}\big)\le 8\log\Big(\frac{3}{\epsilon}\Big)\,,
$$
which gives~\eqref{eq:psisimp}. When $\epsilon = 0$, since $\E \ \psi(|\wh{\mu}^{(1)}_{t}|)= v/t$, Equation \eqref{eq:psiepszero0}
directly gives \eqref{eq:psiepszero}.

%
%



Next, we turn to the the slightly more sophisticated potential function $\psi(x) = \frac{x^2}{\log(4 x/\epsilon)}$. Observe that for any $x \ge 0$,
$$
\psi'(x)=\frac{2x}{\log(4x/\epsilon)}-\frac{x}{\log^2(4x/\epsilon)}\le \frac{2x}{\log(4x/\epsilon)}\,.
$$
Therefore, for $\epsilon \in (0,1]$, the integral in~\eqref{EQ:mainUB} is bounded from above by
\begin{align*}
 \int_{\epsilon/2}^{+\infty} \frac{4x}{\log(4x/\epsilon)[e^{\frac{x^2}{2}}-1]} dx&\le  \int_{\epsilon/2}^{1} \frac{8}{x\log(4x/\epsilon)} dx+ \int_{1}^{\infty} 9e^{-\frac{x^2}{2}} dx\\
 &\le 8\log\log(4/\epsilon)-8\log\log 2+4\\
 &\le 8\log\log(4/\epsilon)+7\,.
\end{align*}
It concludes the proof of~\eqref{eq:psilog}.

\section{Lower bounds} \label{sec:LB}
We conclude our study of bounded regret in stochastic multi-armed bandits with three different lower bounds. For simplicity, we phrase these results for the simple two-armed case. First we show with Theorem \ref{th:LB1} that if one knows both $\mu^{(\star)}$ and $\Delta$, then the best attainable regret is of order $1/\Delta$, which matches (up to a numerical constant) the result of Theorem \ref{th:alg1}. Next we show in Theorem \ref{th:LB2} that the sole knowledge of $\Delta$ leads to a lower bound of order $\log(n \Delta^2) / \Delta$. This theorem implies that the bounds of \cite{AB09}, \cite{AueOrt10} and \cite{PerRig11} exhibit a tight dependence in $\Delta$ (for the two-armed case), unlike the famous result of \cite{LR85}. 
Moreover, compared to the proof of \cite{LR85}, our approach is (i) much simpler, (ii) non-asymptotic
and (iii) it is not limited to a certain class of policies. 
Finally we show in Theorem~\ref{th:LB3} that if one only knows $\mu^{(\star)}$ then a regret of order $\frac{\log(n)}{\Delta}$ is unavoidable (for some value of $\Delta$).
\newline

Our proof strategy consists in rephrasing arm selection as a hypothesis testing problem, and then use well-known lower bounding techniques for the minimax risk of hypothesis testing. For instance, the proof of Theorem \ref{th:LB1} and Theorem \ref{th:LB2} builds upon the following result; see \cite[Chaper 2]{Tsy09} for a proof, or Lemma \ref{lem:2} below with $\lambda$ chosen to be a Dirac mass at $1$. 
Recall that the Kullback-Leibler divergence between two positive measures $\rho, \rho'$ with $\rho'$ absolutely continuous with respect to $\rho$, is defined as
\begin{equation*}
\KL(\rho,\rho') = \int \log\left(\frac{d\rho}{d\rho'}\right) d\rho = 
   \E_{X \sim \rho} \log\left(\frac{d\rho}{d\rho'}(X)\right)\,.
\end{equation*}

\begin{lemma} \label{lem:1}
Let $\rho_0, \rho_1$ be two probability distributions supported on some set $\cX$, with $\rho_1$ absolutely continuous with respect to $\rho_0$. Then for any measurable function $\psi : \cX \rightarrow \{0,1\}$, one has
$$\P_{X \sim \rho_0}(\psi(X) = 1) + \P_{X \sim \rho_1}(\psi(X) = 0) \geq \frac12 \exp\left(- \KL(\rho_0, \rho_1) \right).$$
\end{lemma}

In this section we denote by $\nu=\nu_1 \otimes \nu_2$ the product distribution that generates the rewards from $\nu_j$ when pulling arm $j \in \{1,2\}$. The regret of a policy that 
observes such rewards is denoted by $R_n(\nu)$. Finally let $\P_{\nu}$ denote the probability associated to $\nu$ and by $\E_{\nu}$ the corresponding expectation.

Hereafter, we favor rewards that are normally 
distributed because they lead to simpler calculations of the KL-divergence. However, our lower bounds remain of the same order for all families of distributions $\{\rho_\mu\}_\mu$ with expected value $\mu$ and such that $\KL(\rho_\mu-\rho_{\mu'})\ge C(\mu-\mu')^2$ for some absolute constant $C>0$. This is the case, for example, of the Bernoulli distribution with parameter $\mu$ as long as $\mu$ remains bounded away from 0 and 1; see, e.g., \cite[Lemma~4.1]{RigZee10}.

The first lower bound illustrates that when one knows the distributions up to a permutation, the best one can hope for is a bounded regret of order $1/\Delta$.
\begin{theorem} \label{th:LB1}
Let $\nu= \cN(0,1) \otimes \cN(- \Delta,1)$ and $\nu' = \cN(-\Delta,1) \otimes \cN(0,1)$. Then for any policy, and 
for every $n \geq 1$,
$$\max\left(R_n(\nu), R_n(\nu')\right) \ge \frac{1}{4\Delta}\,.$$
\end{theorem}
\begin{proof}
In this proof we assume that the policy has access to $t$ rewards from each arm at time step $t$. Clearly this full information setting is simpler than the bandit setting, and thus a lower bound for the former implies one for the latter. Using Lemma \ref{lem:1} as well as straightforward computations one obtains
\begin{align*}
&\max\left(R_n(\nu), R_n(\nu')\right)  \geq  \frac{1}{2} \left( R_n(\nu) + R_n(\nu') \right)  = \frac{\Delta}{2} \sum_{t=1}^n \left(\P_{\nu}(I_t = 2) + \P_{\nu'}(I_t = 1) \right) \\
&\geq  \frac{\Delta}{4} \sum_{t=1}^n \exp(- \KL(\nu^{\otimes t}, \nu'^{\otimes t}) ) =  \frac{\Delta}{4} \sum_{t=1}^n \exp(- t \Delta^2) \ge \frac{1}{4\Delta}\,.
\end{align*}
\end{proof}
The above theorem ensures that the regret bound of Theorem~\ref{th:alg1} has the correct dependence in $\Delta$. 
This is quite surprising as the original bound of~\cite{LR85} indicates that without the knowledge of $\mu^{(\star)}$ and $\Delta$, one can incur a regret that diverges to infinity at a logarithmic rate. The next result shows that this logarithmic regret already appears when one does not know the value of $\mu^{(\star)}$. Thus the knowledge of $\Delta$ without the knowledge of $\mu^{(\star)}$ is not sufficient to obtain a bounded regret. Moreover, the following lower bound matches the upper bounds (for the two-armed case) of~\cite{AB09}, \cite{AueOrt10} and~\cite{PerRig11}, thus proving their optimality.
\begin{theorem} \label{th:LB2}
Let $\nu= \delta_0 \otimes \cN(- \Delta,1)$ and $\nu' = \delta_0 \otimes \cN(\Delta,1)$. Then for any policy, and any $n \geq 1$,
$$\max\left(R_n(\nu), R_n(\nu')\right) \geq \frac{\log(n \Delta^2 / 2)}{4 \Delta} .$$
\end{theorem}

\begin{proof}
First note that 
$$\max\left(R_n(\nu), R_n(\nu')\right) \geq R_n(\nu) \ge \Delta \E_{\nu} T_2(n) .$$
Furthermore, denoting by $\nu_t$ (respectively $\nu'_t$) the law of the observed rewards up to time $t$ under $\nu$ (respectively under $\nu'$), and following the same computations than in the previous proof, one also obtains 
$$\max\left(R_n(\nu), R_n(\nu')\right) \geq \frac{\Delta}{4} \sum_{t=1}^n \exp(- \KL(\nu_t, \nu'_t) ).$$
Since under $\nu$, arm 1 is uninformative, it follows from basic calculation that
$$
\KL(\nu_t, \nu'_t) = 2 \Delta^2 \E_{\nu} T_2(t)\,.
$$
The above three displays yield
\begin{align*}
\max\left(R_n(\nu), R_n(\nu')\right) & \geq  \frac{\Delta}{2} \left(\E_{\nu} T_2(n) + \frac{n}{4} \exp(- 2 \Delta^2 \E_{\nu} T_2(n) )  \right) \\
& \geq  \min_{x \in [0,n]} \frac{\Delta}{2} \left(x + \frac{n}{4} \exp(- 2 \Delta^2 x ) \right)\\
& \geq  \frac{\log(n \Delta^2 / 2)}{4 \Delta} .
\end{align*}
\end{proof}

Finally we prove that the knowledge of $\mu^{(\star)}$ without the knowledge of $\Delta$ is not sufficient either to obtain a bounded rescaled regret $\Delta R_n$. This result is more difficult, and falls within the more general topic of lower bounds for adaptive rates. First we need to generalize Lemma \ref{lem:1} to deal with both a composite alternative, and a rescaled risk. The proof of this result is standard and postponed to the appendix.

\begin{lemma} \label{lem:2}
Let $\rho_0$ and $\rho_{\Delta}, \Delta \in \R$ be probability distributions supported on some set $\cX$, with $\rho_{\Delta}$ absolutely continuous with respect to $\rho_0$. Let $\lambda$ be a finite positive measure on $\R$. Then for any measurable function $\psi : \cX \rightarrow \{0,1\}$, one has
$$\P_{X \sim \rho_0}(\psi(X) = 1) + \int \Delta \P_{X \sim \rho_{\Delta}}(\psi(X) = 0) d\lambda(\Delta) \geq \frac1{C_\lambda} \exp\left(- \KL\left(\rho_0,\bar \rho\right) \right) \,,
$$
where $\bar \rho$ is the positive measure on $\cX$ defined by $\bar \rho= \int \Delta \rho_{\Delta} d\lambda(\Delta)$ and $C_\lambda =1+ \int \Delta d\lambda(\Delta)$.
\end{lemma}

Note that $\int \Delta \rho_{\Delta} \lambda(\Delta)$ is not a probability distribution, 
however it is a positive measure thus the Kullback-Leibler divergence in the above lemma is well-defined.

\begin{theorem} \label{th:LB3}
Let $\nu_0 = \cN(0,1) \otimes \cN(-1,1)$, and $\nu_{\Delta} = \cN(-\Delta,1) \otimes \cN(0,1)$, $\Delta \in (0,1]$. Then for any policy, and any $n \geq 1$,
$$\max\left(R_n(\nu_0), \sup_{\Delta \in (0,1]} \Delta R_n(\nu_{\Delta})\right) \geq \frac{1}{2} \log(n / 139) .$$
\end{theorem}
Theorem~\ref{th:LB3} can be read as follows: for any policy, and any $n \geq 1$, there exists $\Delta \in (0,1]$ and a problem instance with gap $\Delta$ and optimal value $\mu^{(\star)} =0$ such that on this problem one has $$R_n \geq \frac{\log(n/139)}{2\Delta}.$$
\begin{proof}
Similarly to the previous proof we define $\nu_{0,t}$ and $\nu_{\Delta, t}$ as the law of the observed rewards up to time $t$.  
 Lemma~\ref{lem:2} yields
\begin{equation} \label{eq:LB3-1}
\max\left(R_n(\nu_0), \sup_{\Delta \in (0,1]} \Delta R_n(\nu_{\Delta})\right) \geq \frac{1}{2C_\lambda} \sum_{t=1}^n \exp\left(- \KL\left(\nu_{0,t}, \int \Delta \nu_{\Delta,t} d\lambda(\Delta) \right) \right).
\end{equation}
For  $\nu \in \{\nu_0, \nu_{\Delta}\}$, define the average rewards for arm $i \in \{1,2\}$ by $\mu_\nu\bi$. Therefore, $\mu_{\nu_0}^{(1)}=\mu_{\nu_\Delta}^{(2)}=0$, $\mu_{\nu_0}^{(2)}=-1$ and $\mu_{\nu_\Delta}^{(1)}=-\Delta$. Recall that a policy $\{I_t\}_{t\ge 1}$ taking values in $\{1,2\}$ generates a sequence of rewards $Y\bpi_t, t\ge 1$ distributed according to $\nu \in \{\nu_0, \nu_\Delta\}$. The joint density (with respect to the Lebesgue measure) $d\nu_t$ of $(Y\bpi_1,\ldots, Y\bpi_t) \in \R^t$, where $\nu \in \{\nu_{\Delta}, \nu_0\}$ can be computed easily using the chain rule for conditional densities. It is given by
$$
d\nu_t=\frac{1}{(2\pi)^{t/2}}\exp\Big(-\frac{1}{2}\sum_{\ell=1}^t(Y^{(I_\ell)}_\ell- \mu_\nu^{(I_\ell)})^2\Big)\,.
$$
Choosing $\nu=\nu_\Delta$ and $\nu=\nu_0$ respectively, it yields
\begin{align*}
&\frac{d \nu_{\Delta,t}}{d \nu_{0,t}}(Y^{(I_1)}_1, \ldots, Y^{(I_t)}_t)=\exp\Big(-\frac{1}{2}\sum_{\ell=1}^t\big[(Y^{(I_\ell)}_\ell- \mu_{\nu_\Delta}^{(I_\ell)})^2 -(Y^{(I_\ell)}_\ell- \mu_{\nu_0}^{(I_\ell)})^2 \big]\Big)\\
&=\exp\Big(-\frac{1}{2}\sum_{\substack{\ell=1\\I_\ell=1 }}^t\big[(Y^{(1)}_\ell+\Delta)^2 -(Y^{(1)}_\ell)^2 \big]-\frac{1}{2}\sum_{\substack{\ell=1\\I_\ell=2 }}^t\big[(Y^{(2)}_\ell)^2 -(Y^{(2)}_\ell+1)^2 \big]\Big)\\
&= \exp \left( - \frac{T^{(1)}}{2} (2 \Delta \hat{\mu}^{(1)} + \Delta^2) + \frac{T^{(2)}}{2} (2 \hat{\mu}^{(2)} + 1) \right)\,,
\end{align*}
where we denote for simplicity $$
T\bi=T_i(t+1)=\sum_{\ell=1}^t \dsone\{I_\ell=i\} \quad \textrm{and} \quad  \hat{\mu}\bi =\hat \mu_{T_i(t)}^{(i)}=\frac{1}{T\bi}\sum_{\substack{\ell=1\\I_\ell=i }}^tY\bi_\ell, \quad i \in \{1,2\}\,.
$$
Dropping the dependency in $(Y^{(I_1)}_1, \ldots, Y^{(I_t)}_t)$ from the notation, it yields
$$
\int \Delta \frac{d \nu_{\Delta,t}}{d \nu_{0,t}} d\lambda(\Delta) = \exp \left( \frac{T^{(2)}}{2} (2 \hat{\mu}^{(2)} + 1) \right) \int \Delta \exp \left( - \frac{T^{(1)}}{2} (2 \Delta \hat{\mu}^{(1)} + \Delta^2) \right) d\lambda(\Delta)\,,
$$
and thus
\begin{align*}
& \KL\left(\nu_{0,t}, \int  \Delta \nu_{\Delta,t} d\lambda(\Delta) \right) \\
& = - \E_{\nu_0} \left(\frac{T^{(2)}}{2} (2 \hat{\mu}^{(2)} + 1) + \log\left(\int  \Delta \exp \left( - \frac{T^{(1)}}{2} (2 \Delta \hat{\mu}^{(1)} + \Delta^2) \right) d\lambda(\Delta)\right) \right) \\
& = \frac12 \E_{\nu_0} T^{(2)} - \E_{\nu_0} \log\left(\int  \Delta \exp \left( - \frac{T^{(1)}}{2} (2 \Delta \hat{\mu}^{(1)} + \Delta^2) \right) d\lambda(\Delta)\right)
\end{align*}
where the last line follows standard computations. Next, it follows from the Cauchy-Schwarz inequality that the function
$$
x \mapsto \log \left( \int_{\Delta} \Delta \exp(\phi(\Delta) x) d\lambda(\Delta) \right)
$$
is convex for any function $\varphi$. Together with the Jensen inequality, it yields
\begin{align*}
\E_{\nu_0} \log&\left(\int \Delta \exp \left( - \frac{T^{(1)}}{2} (2 \Delta \hat{\mu}^{(1)} + \Delta^2) \right) d\lambda(\Delta)\right)\\
& \geq  \log\left(\int \Delta \exp \left( - \E_{\nu_0} \frac{T^{(1)}}{2} (2 \Delta \hat{\mu}^{(1)} + \Delta^2) \right) d\lambda(\Delta)\right) \\
& =  \log\left(\int  \Delta \exp \left( - \frac{\E_{\nu_0} T^{(1)}}{2} \Delta^2\right) d\lambda(\Delta)\right) 
\end{align*}
Define $\tau=\E_{\nu_0} T^{(1)}$ and  let $\lambda$ be the uniform distribution on $[0,1/\sqrt{\tau}]$.  Since $ue^{-u^2/2}\ge u/2$ for $0\le u \le 1$, it yields 
\begin{align*}
\int \Delta \exp \left( - \frac{\E_{\nu_0} T^{(1)}}{2} \Delta^2\right) d\lambda(\Delta) &=\frac{1}{\sqrt{\tau}}\int_0^1 u \exp(-u^2/2)du \ge  \frac{1}{4\sqrt{\tau}}\,,
\end{align*}

%
%
Thus we have proved that 
\begin{align*}
\KL\left(\nu_{0,t}, \int_{0}^1 \Delta \nu_{\Delta,t} d\Delta \right) & \leq  \frac{1}{2} \E_{\nu_0} T^{(2)} + \log(4\sqrt{\E_{\nu_0} T^{(1)}}) \\
& \leq  \frac{1}{2} \E_{\nu_0} T_2(n) + \frac12 \log(16n).
\end{align*}
Plugging this into \eqref{eq:LB3-1} one obtains 
\begin{align*}
\max\left(R_n(\nu_0), \sup_{\Delta \in (0,1]} \Delta R_n(\nu_{\Delta})\right) &\geq \frac{\sqrt{n}}{8C_\lambda} \exp\left(- \frac{1}{2} \E_{\nu_0} T_2(n) \right)\\
&\ge \frac{\sqrt{n}}{16} \exp\left(- \frac{1}{2} \E_{\nu_0} T_2(n) \right)\,,
\end{align*}
where we use the fact that $\tau \ge 1$, which implies $C_\lambda \le3/2 \le2$.
On the other hand one also has
$$
R_n(\nu_0) \geq \E_{\nu_0} T_2(n)\,
$$
Therefore
\begin{align*}
\max\left(R_n(\nu_0), \sup_{\Delta \in (0,1]} \Delta R_n(\nu_{\Delta})\right)& \geq \min_{x \in [0,n]} \frac{1}{2}\Big(x+\frac{\sqrt{n}}{16}\exp(-x/2)  \Big)\\
&=\frac{1}{2}\log(n/139)\,.
\end{align*}
\end{proof}

Theorem~\ref{th:LB2} and~\ref{th:LB3} have important consequences on the \emph{exploration-exploitation tradeoff} mentioned in the introduction.  Indeed, consider the full information case where at each round, the agent observes the reward of both arms. In this case, it is not hard to see that the policy that indicates to pull the arm with the best average reward has bounded regret of order $1/\Delta$. Therefore, the knowledge of $\Delta$ or $\mu^{(\star)}$ alone does not alleviate the price for exploration. However, when both are known, it vanishes (see Theorem~\ref{th:alg1}).

\medskip

{\bf Acknowledgments.} We are indebted to Alexander Goldenshluger for bringing the reference \cite{LR84book} to our attention.

%
%
%

\appendix
\section{Proof of Lemma 7}
Throughout the proof, Radon-Nikodym derivatives over $\cX$ are taken with respect to a common but unspecified reference measure. It does not enter our final result. It follows from Fubini's Theorem that
\begin{align*}
& \P_{X \sim \rho_0}(\psi(X) = 1) + \int \Delta \P_{X \sim \rho_{\Delta}}(\psi(X) = 0) d\lambda(\Delta)   \\
&=  \int_{\psi = 1} d \rho_0 + \int\left( \int_{\psi = 0} \Delta d \rho_{\Delta} \right) d\lambda(\Delta) \\
& = \int_{\psi = 1} d \rho_0 + \int_{\psi = 1} d\bar \rho \\
& = \int_{\psi = 0} d \rho_0 + \int_{\psi = 1} \frac{d\bar \rho }{d \rho_0}  d\rho_0
\end{align*}
Furthermore the last expression is clearly minimized for $\psi(x) = \dsone\left\{\frac{d\bar \rho}{d \rho_0}(x) > 1 \right\}$\,. It yields
\begin{align*}
  \int_{\psi = 1} d \rho_0 + \int_{\psi = 0} \frac{d\bar \rho }{d \rho_0}  d\rho_0& \geq \int_{\frac{d\bar \rho}{d \rho_0} > 1} d \rho_0 + \int_{\frac{d\bar \rho}{d \rho_0} \le 1} \frac{d\bar \rho}{d \rho_0}  d\rho_0(x)  \\
& = \int_{\frac{d\bar \rho}{d \rho_0}  > 1} d \rho_0 + \int_{\frac{d\bar \rho}{d \rho_0}\le 1 } d\bar{\rho} \\
& = \int \min\left( d \rho_0, d\bar{\rho} \right)\,.
\end{align*}
Note that the latter quantity is often referred to as \emph{Hellinger affinity} and does not depend on the reference measure on $\cX$; see, e.g., \cite{Tsy09}, Chapter 2.
Now using the Cauchy-Schwarz inequality and the fact that 
$$
\int \min\left( d \rho_0, d\bar{\rho} \right) + \int \max\left( d \rho_0, d\bar{\rho} \right) = C_\lambda\,,
$$ 
we get
\begin{align*}
\left( \int \sqrt{d\bar{\rho} d\rho_0}\right)^2 & =  \left( \int \sqrt{\min(d\bar{\rho},d\rho_0) \max(d\bar{\rho},d\rho_0)}\right)^2 \\
& \leq  \left( \int_x \min(d\bar{\rho},d\rho_0) \right) \left( \int_x \max(d\bar{\rho},d\rho_0) \right) \\
& \leq  C_\lambda \int_x \min(d\bar{\rho},d\rho_0) .
\end{align*}
The above three displays together yield
$$
\P_{X \sim \rho_0}(\psi(X) = 1) + \int_{\Delta} \Delta \P_{X \sim \rho_{\Delta}}(\psi(X) = 0) d\lambda(\Delta) \geq \frac{1}{C_\lambda} \left( \int \sqrt{d\bar{\rho} d\rho_0}\right)^2 .$$
To complete the proof, observe that the Jensen inequality yields 
\begin{align*}
\Big( \int \sqrt{d\bar{\rho} d\rho_0}\Big)^2 & =  \Big( \int \sqrt{\frac{d\bar{\rho}}{d\rho_0}} d\rho_0\Big)^2\\
& =  \exp\Big[2 \log \Big( \int \sqrt{\frac{d\bar{\rho}}{d\rho_0}} d\rho_0 \Big) \Big] \\
& \geq  \exp\Big[2 \int \log \Big( \sqrt{\frac{d\bar{\rho}}{d\rho_0}} \Big) d\rho_0 \Big]\\
& =  \exp[ - \KL(\rho_0, \bar{\rho}) ] .
\end{align*}

\end{document}